\documentclass{amsart}

\usepackage{amsthm,amssymb,amsmath}

\newtheorem{theorem}{Theorem}[section]
\newtheorem{lemma}[theorem]{Lemma}
\newtheorem{proposition}[theorem]{Proposition}

\theoremstyle{definition}
\newtheorem{definition}[theorem]{Definition}
\newtheorem{remark}[theorem]{Remark}
\newtheorem{question}[theorem]{Question}

\newcommand{\ZZ}{{\mathbb Z}}
\newcommand{\A}{{\mathcal A}}
\newcommand{\Aut}{\mathsf{Aut}}
\newcommand{\GL}{\mathsf{GL}}
\renewcommand{\vec}{\mathbf}
\newcommand{\Mod}{\mbox{Mod}}
\newcommand{\Div}{\mbox{Div}}
\newcommand{\Aff}{\mbox{Aff}}
\newcommand{\M}{{\mathcal M}}

\begin{document}

\title{The conjugacy problem in automaton groups \\ is not solvable}

\author{Zoran {\v{S}}uni\'c}
\address{Dept. of Mathematics, Texas A\&M Univ. MS-3368, College Station, TX 77843-3368, USA}

\author{Enric Ventura}
\address{Dept. Mat. Apl. III, Universitat Polit\`ecnica de Catalunya, Manresa, Barcelona, Catalunya}

\begin{abstract}
(Free-abelian)-by-free, self-similar groups generated by finite self-similar sets of tree automorphisms and having
unsolvable conjugacy problem are constructed. Along the way, finitely generated, orbit undecidable, free subgroups of $\GL_d(\ZZ)$, for $d \geqslant 6$, and $\Aut(F_d)$, for $d \geqslant 5$, are constructed as well. 
\end{abstract}


\def\subjclassname{\textup{2010} Mathematics Subject Classification}
\subjclass[2010]{20E8; 20F10}

\maketitle


\section{Introduction}

The goal of this paper is to prove the following result. 

\begin{theorem}\label{t:main}
There exist automaton groups with unsolvable conjugacy problem. 
\end{theorem}

The question on solvability of the conjugacy problem was raised in 2000 for the class of self-similar groups
generated by finite self-similar sets (automaton groups) by Grigorchuk, Nekrashevych and Sushchanski{\u\i}~\cite{grigorchuk-n-s:automata}. Note that the word problem is solvable for all groups in this class by a rather straightforward algorithm running in exponential time. Moreover, for an important subclass consisting of finitely generated, contracting groups the word problem is solvable in polynomial time. Given that our examples contain free nonabelian subgroups and that contracting groups, as well as the groups $\mathsf{Pol}(n)$, $n \geq0$, do not contain such subgroups (see~\cite{nekrashevych:free-subgroups} and~\cite{sidki:pol}), the following question remains open. 

\begin{question}\label{q:open}
Is the conjugacy problem solvable in

(i) all finitely generated, contracting, self-similar groups?

(ii) the class of automaton groups in $\mathsf{Pol}(n)$, for $n \geqslant 0$?
\end{question}

There are many positive results on the solvability of the conjugacy problem in
automaton groups close to the first Grigorchuk group~\cite{grigorchuk:burnside} and the Gupta-Sidki examples~\cite{gupta-s:burnside}. The conjugacy problem was solved for the first Grigorchuk group independently by Leonov~\cite{leonov:conjugacy} and Rozhkov~\cite{rozhkov:conjugacy}, and for the Gupta-Sidki examples by Wilson and Zaleskii~\cite{wilson-z:conjugacy}. Grigorchuk and Wilson~\cite{grigorchuk-w:conjugacy} showed that the problem is solvable in all subgroups of finite index in the first Grigorchuk group. In fact, the results  in~\cite{leonov:conjugacy,wilson-z:conjugacy,grigorchuk-w:conjugacy} apply to certain classes of groups that include the well known examples we explicitly mentioned. In a recent work Bondarenko, Bondarenko, Sidki and Zapata~\cite{bondarenko-b-s-z:conjugacy} showed that the conjugacy problem is solvable in $\mathsf{Pol}(0)$. Lysenok, Myasnikov, and Ushakov provided the first, and so far the only, significant result on the complexity of the conjugacy problem in automaton groups by providing a polynomial time solution for the first Grigorchuk group~\cite{lysenok-m-u:grigorchuk-cp}.

The strategy for our proof of Theorem~\ref{t:main} is as follows. First, we observe the following consequence of a result by Bogopolski, Martino and Ventura~\cite{bogopolski-m-v:cp}. 

\begin{proposition}\label{gamma}
Let $H$ be a finitely generated group, and $\Gamma$ a finitely generated subgroup of $\Aut (H)$. If $\Gamma \leqslant
\Aut (H)$ is orbit undecidable then $H\rtimes \Gamma$ has unsolvable conjugacy problem.
\end{proposition}

Since, for $d\geqslant 4$, examples of finitely generated orbit undecidable subgroups $\Gamma$ in $GL_d(\mathbb{Z})$ are provided in~\cite{bogopolski-m-v:cp}, we obtain the existence of groups of the form $\mathbb{Z}^d \rtimes \Gamma$ with unsolvable conjugacy problem. Finally, using techniques of Brunner and Sidki~\cite{brunner-s:glnz}, we prove the following result, which implies Theorem~\ref{t:main}. 

\begin{theorem}\label{cor}
Let $\Gamma$ be an arbitrary finitely generated subgroup of $GL_d(\mathbb{Z})$. Then, $\mathbb{Z}^d \rtimes \Gamma$ is an automaton group.
\end{theorem}

The examples of finitely generated orbit undecidable subgroups of $\GL_d(\ZZ)$, for $d\geqslant 4$ given in~\cite{bogopolski-m-v:cp} are based on Mikhailova's construction and are not finitely presented. By modifying the construction in~\cite{bogopolski-m-v:cp}, at the cost of increasing the dimension by 2, we determine finitely generated, orbit undecidable, free subgroups of $GL_d(\mathbb{Z})$, for $d\geqslant 6$. Note that, by~\cite[Proposition 6.9.]{bogopolski-m-v:cp} and the Tits Alternative~\cite{tits:alternative}, every orbit undecidable subgroup $\Gamma$ of $\GL_d(\ZZ)$ contains free nonabelian subgroups. By using the same technique (see Proposition~\ref{p:general-free}) we also construct finitely generated, orbit undecidable, free subgroups of $\Aut(F_d)$, for $d \geq 5$, answering Question~6 raised in~\cite{bogopolski-m-v:cp}.

\begin{proposition}\label{p:free-orbit-undecidable}
(a) For $d\geqslant 6$, the group $\GL_d(\ZZ)$ contains finitely generated, orbit undecidable, free subgroups. 

(b) For $d\geqslant 5$, the group $\Aut(F_d)$ contains finitely generated, orbit undecidable, free subgroups. 
\end{proposition}

This allows us to deduce the following strengthened version of
Theorem~\ref{t:main}.

\begin{theorem}\label{gros}
For every $d\geqslant 6$, there exists a finitely presented group $G$ simultaneously satisfying the following three
conditions:

i) $G$ is an automaton group,

ii) $G$ is $\ZZ^d$-by-(f.g.-free) (in fact, $G=\ZZ^d \rtimes_\phi F_m$, with injective action $\phi$),

iii) $G$ has unsolvable conjugacy problem.
\end{theorem}

\section{Orbit undecidability}\label{oud}

The main result in~\cite{bogopolski-m-v:cp} can be stated in the following way.

\begin{theorem}[Bogopolski, Martino, Ventura~\cite{bogopolski-m-v:cp}]\label{bmv}
Let $G=H\rtimes F$ be a semidirect product (with $F$, $H$, and so $G$, finitely generated) such that

(i) the conjugacy problem is solvable in $F$,

(ii) for every $f \in F$, $\langle f \rangle$ has finite index in the centralizer $C_F(f)$ and there is
an algorithm that, given $f$, calculates coset representatives for $\langle f \rangle$ in $C_F(f)$,

(iii) the twisted conjugacy problem is solvable in $H$.

\noindent
Then the following are equivalent:

(a) the conjugacy problem in $G$ is solvable,

(b) the conjugacy problem in $G$ restricted to $H$ is solvable,

(c) the action group $\{\lambda_g \mid g \in G\} \leqslant \Aut(H)$ is orbit decidable, where $\lambda_g$ denotes the
right conjugation by $g$, restricted to $H$. \qed
\end{theorem}

The \emph{conjugacy problem in $G$ restricted to $H$} asks if, given two elements $u$ and $v$ in $H$, there exists an
element $g$ in $G$ such that $u^g=v$. The \emph{orbit problem} for a subgroup $\Gamma$ of $\Aut(H)$ asks if, given $u$ and $v$ in $H$, there is an automorphism $\gamma$ in $\Gamma$ such that $\gamma(u)$ is conjugate to $v$ in $H$; we say that $\Gamma$ is \emph{orbit decidable (resp. undecidable)} if the orbit problem for $\Gamma$ is solvable (resp. unsolvable). Finally, the \emph{twisted conjugacy problem} for a group $H$ asks if, given an automorphism $\varphi \in
\Aut (H)$ and two elements $u,v\in H$, there is $x\in H$ such that $v=\varphi(x)^{-1}ux$.

The implications $(a) \Rightarrow (b) \Leftrightarrow (c)$ in Theorem~\ref{bmv} are clear from the definitions, and do not require most of the hypotheses (as indicated in~\cite{bogopolski-m-v:cp}, the only relevant implication is $(c)
\Rightarrow (a)$). Proposition~\ref{gamma}, which is needed for our purposes, is an obvious corollary. 

\begin{proposition}\label{p:general-free}
Let $G$ be a group and $H$ and $K$ be subgroups of $G$ such that 

(i) $G = \langle H, K \rangle$, 

(ii) the free group $F_2$ of rank 2 is a subgroup of $\Aut(K)$,

(iii) there exists a finitely generated orbit undecidable subgroup $\Gamma \leqslant \Aut(H)$, 

(iv) every pair of automorphisms $\alpha \in \Aut(H)$ and $\beta \in \Aut(K)$ has a (necessarily unique) common extension to an automorphism of $G$, and  

(v) two elements of $H$ are conjugate in $G$ if and only if they are conjugate in $H$. 

Then, $\Aut(G)$ contains finitely generated, orbit undecidable, free subgroups. 
\end{proposition}

\begin{proof}
Let $\Gamma = \langle g_1,\ldots,g_m \rangle$ be an orbit undecidable subgroup of $\Aut(H)$ and $F=\langle f_1,\dots,f_m \rangle$ be a free subgroup of rank $m$ of $\Aut(K)$. For, $i=1,\dots,m$, let $s_i$ be the common extension of $g_i$ and $f_i$ to an automorphism of $G$ and let $\Gamma' = \langle s_1,\dots,s_m \rangle \leqslant \Aut(G)$. Since $F$ is free of rank $m$, so is $\Gamma'$.  

Moreover, $\Gamma'$ is orbit undecidable subgroup of $\Aut(G)$. Indeed, for $u,v \in H$, 
\begin{align*}
 (\exists \gamma' \in \Gamma')(\exists t' \in G)~\gamma'(u)=v^{t'} 
   &\iff (\exists \gamma \in \Gamma)(\exists t' \in G)~\gamma(u)=v^{t'} \iff \\
   &\iff (\exists \gamma \in \Gamma)(\exists t \in H)~\gamma(u)=v^{t}
\end{align*}
The second equivalence follows from (v), since $\gamma(u), v \in H$. The first comes from the construction, since, for every group word $w(x_1,\dots,x_m)$, the automorphisms $\gamma'=w(s_1,\dots,s_m) \in \Gamma'$ and $\gamma=w(g_1,\dots,g_m) \in \Gamma$ agree on $H$. Therefore, the orbit problem for the instance $u,\, v\in H$ with respect to $\Gamma \leqslant \Aut(H)$ is equivalent to the orbit problem for the instance $u,v\in H\leqslant G$ with respect to $\Gamma' \leqslant \Aut(G)$, showing that an algorithm that would solve the orbit problem for $\Gamma'$ could be used to solve the orbit problem for $\Gamma$ as well. Thus $\Gamma'$ is orbit undecidable. 
\end{proof}

\begin{proof}[Proof of Proposition~\ref{p:free-orbit-undecidable}] 
(a) For $d \geq 6$, let $G=\ZZ^d$, $H=\ZZ^{d-2}$, $K=\ZZ^2$, and $G=H \oplus K$. All requirements of Proposition~\ref{p:general-free} are satisfied. In particular, (iii) holds by~\cite[Proposition~7.5]{bogopolski-m-v:cp}, and (v) holds since conjugacy is the same as equality in both $H$ and $G$. 

(b) For $d \geq 5$, let $G=F_d$, $H=F_{d-2}$, $K=F_2$, and $G=H \ast K$. All requirements of Proposition~\ref{p:general-free} are satisfied. In particular, (iii) holds by~\cite[Subsection~7.2]{bogopolski-m-v:cp}, and (v) holds since the free factor $H$ is malnormal in $G$. 
\end{proof}

\section{Self-similar groups and automaton groups}\label{ssg}

Let $X$ be a finite alphabet on $k$ letters. The set $X^*$ of words over $X$ has the structure of a rooted $k$-ary tree in which the empty word is the root and each vertex $u$ has $k$ children, namely the vertices $ux$, for $x$ in $X$. Every tree automorphism fixes the root and permutes the words of the same length (constituting the levels of the rooted tree) while preserving the tree structure.

Let $g$ be a tree automorphism. The action of $g$ on $X^*$ can be decomposed as follows. There is a permutation $\pi_g$ of $X$, called the \emph{root permutation} of $g$,  determined by the permutation that $g$ induces on the subtrees below the root (the action of $g$ on the first letter in every word), and tree automorphisms $g|_x$, for $x$ in $X$, called the \emph{sections} of $g$, determined by the action of $g$ within these subtrees (the action of $g$ on the rest of the word behind the first letter). Both the root permutation and the sections are uniquely determined by the equality
\begin{equation}\label{e:gxw}
 g(xw) = \pi_g(x)g|_x(w),
\end{equation}
for $x$ in $X$ and $w$ in $X^*$.

A group or a set of tree automorphisms is \emph{self-similar} if it contains all sections of all of its
elements. A \emph{finite automaton} is a finite self-similar set. A group $G(\A)$ of tree automorphisms generated by a finite self-similar set $\A$ is itself self-similar and it is called an \emph{automaton group} (realized or generated by the automaton $\A$). The elements of the automaton are often referred to as \emph{states} of the automaton and the automaton is said to operate on the alphabet $X$.

The \emph{boundary} of the tree $X^*$ is the set $X^\omega$ of right infinite words $x_1x_2x_3\cdots$. The tree structure induces a metric on $X^\omega$ inducing the Cantor set topology. The metric is given by $d(u,v) = \frac{1}{2^{|u \wedge v|}}$, for $u \neq v$, where $|u \wedge v|$ denotes the length of the longest common prefix of $u$ and $v$. The group of isometries of the boundary $X^\omega$ and the group of tree automorphism of $X^*$ are canonically isomorphic. Every isometry induces a tree automorphism by restricting the action on finite prefixes, and every tree automorphism induces an isometry on the boundary through an obvious limiting process. The decomposition formula~(\ref{e:gxw}) for the action of tree automorphisms is valid for boundary isometries as well ($w$ is any right infinite word in this case).

\section{Automaton groups with unsolvable conjugacy problem}\label{feina}

Let $\M=\{ M_1,\ldots,M_m \}$ be a set of integer $d\times d$ matrices with non-zero determinants. Let $n\geqslant 2$ be relatively prime to all of these determinants (thus, each $M_i$ is invertible over the ring $\ZZ_n$ of $n$-adic integers. For an integer matrix $M$ and an arbitrary vector $\vec{v}$ with integer coordinates, consider the
invertible affine transformation $M_\vec{v} \colon \ZZ_n^d \to \ZZ_n^d$ given by $M_\vec{v}(\vec{u})=\vec{v}+M\vec{u}$,
and let
\[
G_{\M, n} =\langle \{M_\vec{v} \mid M \in \M, \ \vec{v} \in \ZZ^d\} \rangle
\]
be the subgroup of $\Aff_d(\ZZ_n)$ generated by all the transformations of the form $M_\vec{v}$, for $M\in \M$ and
$v\in \ZZ^d$. Denote by $\tau_\vec{v}$ the translation $\ZZ_n^d \to \ZZ_n^d$, $\vec{u} \mapsto
\vec{v}+\vec{u}$, and by $\vec{e}_i$ the $i$-th standard basis vector. Since $M_{\vec{v}} =\tau_{\vec{v}} M_{\vec{0}}$, we have
\begin{equation}\label{e:generators}
 G_{\M,n} =\langle \{M_\vec{0} \mid M \in \M\} \cup \{\tau_{\vec{e}_i} \mid i=1,\ldots,d\} \rangle \leqslant \Aff_d(\ZZ_n). 
\end{equation} 

\begin{lemma}\label{inv}
If all matrices in $\M$ are invertible over $\ZZ$, then $G_{\M,n} \cong \ZZ^d \rtimes \Gamma$,
where $\Gamma=\langle \M \rangle \leqslant \GL_d(\ZZ)$; in particular, $G_{\M,n}$ does not depend on $n$.
\end{lemma}

\begin{proof}
If $M$ is an invertible matrix over $\ZZ$, and $v\in \ZZ^d$, then $M_{\vec{v}}\in \Aff_d(\ZZ_n)$ restricts
to a bijective affine transformation $M_{\vec{v}}\in \Aff_d(\ZZ)$. Hence, we can view $G_{\M,n}$ as a
subgroup of $\Aff_d(\ZZ)$ and, in particular, it is independent from $n$; let us denote it by $G_{\M}$.

Clearly, the subgroup of translations $T=\langle \tau_{\vec{e}_1}, \ldots, \tau_{\vec{e}_d}\rangle$ of $G_{\M}$ is free abelian of rank $d$, $T\simeq \ZZ^d$. Since each of the transformations $M_\vec{0}$, for $M \in \M$,  acts on $\ZZ^d$ by multiplication by $M$, the subgroup $\langle M_\vec{0} \mid M \in \M \rangle$ of $G_\M$ is isomorphic to $\Gamma$ and may be safely identified with it. The subgroups $T$ and $\Gamma$ intersect trivially, since every nontrivial element of $T$ moves the zero vector in $\ZZ^d$, while no element of $\Gamma$ does. For $M \in \M \cup \M^{-1}$ (where $\M^{-1}$ is the set of integer matrices inverse to the matrices in $\M$), $j=1,\ldots,d$, and $\vec{u} \in \ZZ^d$,
 $$
M_\vec{0} \tau_{\vec{e}_j} (M_\vec{0})^{-1}(\vec{u}) = M_\vec{0} \tau_{\vec{e}_j} (M^{-1}\vec{u}) = M_\vec{0}(\vec{e}_j
+ M^{-1}\vec{u}) = M\vec{e_j} + \vec{u} =
 $$
 $$
=\tau_{\vec{e}_1}^{m_{1,j}}\tau_{\vec{e}_2}^{m_{2,j}}\cdots \tau_{\vec{e}_d}^{m_{d,j}}(\vec{u}),
 $$
where $m_{i,j}$ is the $(i,j)$-entry of $M$. Therefore, 
for $M \in \M \cup \M^{-1}$ and $j=1,\dots,d$,
\begin{equation}\label{e:relation}
M_\vec{0} \tau_{\vec{e}_j} (M_\vec{0})^{-1} = \tau_{\vec{e}_1}^{m(1,j)}\tau_{\vec{e}_2}^{m(2,j)}\cdots
\tau_{\vec{e}_d}^{m(d,j)}. 
\end{equation}
It follows that the subgroup $T\cong \ZZ^d$ is normal in $G_\M$ and $G_\M \cong \ZZ^d \rtimes \Gamma$. 
\end{proof}

\begin{remark}
The equality~(\ref{e:relation}) is correct (over $\ZZ_n$) for any integer matrix with non-zero determinant relatively prime to $n$. When $\M=\{M\}$ consists of a single $d\times d$ integer matrix $M=(m_{i,j})$ of infinite order and determinant $k\neq 0$ relatively prime to $n$, the multiplication by $M$ embeds $\ZZ^d$ into an index $|k|$ subgroup of $\ZZ^d$ and $G_{\M,n}$ is the ascending HNN extension of $\ZZ^d$ by a single stable letter (see~\cite{bartholdi-s:bs}), i.e.,
\[
G_{\M,n} \cong \langle \ a_1,\ldots,a_d, t \mid [a_i,a_j]=1,~ta_jt^{-1} = a_1^{m_{1,j}}\cdots
a_d^{m_{d,j}},~\mbox{for}~1\leqslant i,j\leqslant d \ \rangle.
\]
\end{remark}

The goal now is to show that the groups $G_{\M,n}$ constructed in this way, can all be realized by finite automata and so, they are automaton groups.

The elements of the ring $\ZZ_n$ may be (uniquely) represented as right infinite words over the alphabet $Y_n=
\{0,\ldots,n-1\}$, through the correspondence
 $$
y_1y_2y_3 \cdots \quad \longleftrightarrow \quad  y_1 + y_2\cdot n + y_3\cdot n^2 + \cdots,
 $$
while the elements of the free $d$-dimensional module $\ZZ_n^d$, viewed as column vectors, may be (uniquely)
represented as right infinite words over the alphabet $X_n =Y_n^d=\{ (y_1,\ldots,y_d)^T \mid y_i \in Y_n, \
i=1,\ldots,d\}$ consisting of column vectors with entries in $Y_n$. Note that $|Y_n |=n$ and $|X_n |=n^d$.

For a vector $\vec{v}$ with integer coordinates define $\Mod(\vec{v})$ and $\Div(\vec{v})$ to be the vectors whose
coordinates are the remainders and the quotients, respectively, obtained by dividing the coordinates of $\vec{v}$ by
$n$, i.e., the unique integer vectors satisfying $\vec{v}=\Mod(\vec{v})+n\Div(\vec{v})$, with $\Mod(\vec{v})\in X_n$. 

\begin{lemma}
For every vector $\vec{v}$ with integer coordinates, and every element $\vec{x}_1\vec{x}_2\vec{x}_3\ldots$ in the free module $\ZZ_n^d$ (where $\vec{x}_1,\vec{x}_2,\vec{x}_3,\ldots$ are symbols in $X_n$), 
\begin{equation}\label{e:Mv}
M_\vec{v}(\vec{x}_1\vec{x}_2\vec{x}_3\cdots) = \Mod(\vec{v}+M\vec{x}_1) +
nM_{\Div(\vec{v}+M\vec{x}_1)}(\vec{x}_2\vec{x}_3\vec{x}_4\cdots). 
\end{equation}
\end{lemma}

\begin{proof}
Indeed, 
\begin{align*}
 M_\vec{v}(\vec{x}_1\vec{x}_2\vec{x}_3\cdots) 
  &= \vec{v}+M\vec{x}_1\vec{x}_2\vec{x}_3\cdots = \vec{v}+M(\vec{x}_1 + n(\vec{x}_2\vec{x}_3\vec{x}_4\cdots)) \\ 
  &= \vec{v}+M\vec{x}_1 + nM\vec{x}_2\vec{x}_3\vec{x}_4\cdots \\   
  &= \Mod(\vec{v}+M\vec{x}_1) + n\Div(\vec{v}+M\vec{x}_1) + nM\vec{x}_2\vec{x}_3\vec{x}_4\cdots \\   
  &= \Mod(\vec{v}+M\vec{x}_1) + n(\Div(\vec{v}+M\vec{x}_1) + M\vec{x}_2\vec{x}_3\vec{x}_4\cdots) \\   
  &= \Mod(\vec{v}+M\vec{x}_1) + nM_{\Div(\vec{v} + M\vec{x}_1)}(\vec{x}_2\vec{x}_3\vec{x}_4\cdots). 
\end{align*} 
\end{proof}

Let $||M||$ be the maximal absolute row sum norm of $M$, i.e. $||M||=\max_i \sum_{j=1}^d |m_{i,j}|$, where $m_{i,j}$ is
the $(i,j)$-entry of $M$. Define $V_M$ to be the finite set of integer vectors $\vec{v}$ for which each coordinate is between $-||M||$ and $||M||-1$, inclusive. Note that $V_M$ is finite and contains $(2||M||)^d$ vectors.

\begin{definition}
For an integer matrix $M$, define an automaton $\A_{M,n}$ operating on the alphabet $X_n$ as follows: the set of
states is $S_{M,n}=\{m_{\vec{v}} \mid \vec{v} \in V_M \}$, and the root permutations and the sections are, for $\vec{x}$ in $X_n$,  defined by
\begin{equation}\label{e:mv-root-section}
 m_\vec{v} (\vec{x}) = \Mod(\vec{v}+M\vec{x}) \qquad \text{and} \qquad
 m_\vec{v}|_{\vec{x}} = m_{\Div(\vec{v}+M\vec{x})}.
\end{equation}
\end{definition}

The automaton $\A_{M,n}$ is well defined (it is easy to show that, for $\vec{v}\in V_M$ and $\vec{x}\in X_n$, the entries of the vector $\vec{v}+M\vec{x}$ are bounded between $-||M||n$ and $||M||n-1$, and hence $\Div(\vec{v}+M\vec{x}) \in V_M$). 

\begin{lemma}\label{l:mvMv}
For every state $m_{\vec{v}}$ of the automaton $\A_{M,n}$, and every element
$\vec{u}=\vec{x}_1\vec{x}_2\vec{x}_3\cdots$ of the free module $\ZZ_n^d$ (i.e. every right infinite word over $X_n$),
\[  m_{\vec{v}}(\vec{u}) = M_\vec{v} (\vec{u}). \]
\end{lemma}

\begin{proof}
Follows directly from the definition of the root permutations and the sections of $m_{\vec{v}}$ in~(\ref{e:mv-root-section}) and equality~(\ref{e:Mv}) describing the action of $M_{\vec{v}}$.
\end{proof}

\begin{definition}
Let $\A_{\M,n}$ be the automaton operating on the alphabet $X_n$ and having $2^d\sum_{i=1}^m ||M_i||^d$ states obtained
by taking the (disjoint) union of the automata $\A_{M_1,n},\ldots,\A_{M_m,n}$.
\end{definition}

\begin{proposition}\label{p:GM}
The group $G_{\M,n}$ can be realized by a finite automaton acting on an alphabet of size $n^d$
and having no more than $2^d\sum_{i=1}^m ||M_i||^d$ states, where $||M_i||$ is the maximum absolute row sum norm of
$M_i$, for $i=1,\ldots,m$.
\end{proposition}

\begin{proof}
The automaton $\A_{\M,n}$ satisfies the required conditions, and generates precisely the group $G_{\M,n}$. This follows directly from~(\ref{e:generators}) and Lemma~\ref{l:mvMv}, once it is observed that $A_{\M,n}$ has enough
states to generate $G_{\M,n}$. However, this is clear, since each of the automata $\A_{M,n}$, for $M\in \M$,
has at least $d+1$ states, $m_\vec{0}, m_{-\vec{e}_1},\ldots,m_{-\vec{e}_d}$, and
$m_\vec{0}(m_{-\vec{e}_j})^{-1}=\tau_{\vec{e}_j}$, for $j=1,\ldots,d$.
\end{proof}

Theorem~\ref{cor} is an immediate corollary of Lemma~\ref{inv}~(ii) and Proposition~\ref{p:GM}. 

\begin{proof}[Proof of Theorem~\ref{gros}]
Let $d\geqslant 6$ and let $F$ be an orbit undecidable, free subgroup of rank $m$ of $\GL_d(\ZZ)$ (such a group exists by Proposition~\ref{p:free-orbit-undecidable}). Let $\M=\{M_1,\ldots,M_m\}$ be a set of invertible integer $d\times d$ matrices generating $F=\langle \M \rangle$. Fix $n\geqslant 2$ and consider the group $G=G_{\M,n}$. By
Proposition~\ref{p:GM}, $G$ is generated by the finite automaton $\A_{\M,n}$, so it is an automaton group. By
Lemma~\ref{inv}~(ii), $G$ does not depend on $n$ and is in fact isomorphic to $\ZZ^d \rtimes F$ (since all matrices in $\M$ are invertible over $\ZZ$); so, it is a $\ZZ^d$-by-free group. Finally, by Proposition~\ref{gamma}, $G=G_{\M,n}$ has unsolvable conjugacy problem.
\end{proof}

Theorem~\ref{t:main} is an immediate corollary of Theorem~\ref{gros}.


\subsection*{Acknowledgments} The authors express their gratitude to CRM at Universit\'e de Montr\'eal and the Organizers of the Thematic Semester in Geometric, Combinatorial, and Computational Group Theory for their hospitality and support in the fall of 2010, when this research was conducted. The first named author was partially supported by the NSF under DMS-0805932 and DMS-1105520. The second one was partial supported by the MEC (Spain) and the EFRD (EC) through projects number MTM2008-01550 and PR2010-0321.


\def\cprime{$'$}


\end{document}